\documentclass{amsart}
\usepackage{xcolor}
\usepackage{hyperref}
\hypersetup{
	colorlinks=true,
	citecolor=blue,
	linkcolor=blue,
	filecolor=magenta,      
	urlcolor=cyan,
}
\usepackage[capitalise,noabbrev]{cleveref}
\usepackage[all, knot]{xy}\SelectTips{eu}{}

\usepackage{amsmath,amscd}
\usepackage{amssymb}
\usepackage{mathtools}
\usepackage{stmaryrd}
\usepackage{url}
\usepackage{tikz-cd}
\usepackage{enumitem}

\usepackage{setspace}
\usepackage{microtype}

\usepackage{times}

\usepackage{tikz}

\usepackage{enumitem,kantlipsum}

\address{Department of Algebra, Faculty of Mathematics and Physics, Charles University in Prague, Sokolovsk\'a 83, 186 75 Praha, Czech Republic}

\email{shaul@karlin.mff.cuni.cz}

\newtheorem{thm}[equation]{Theorem}
\newtheorem*{thm*}{Theorem}
\newtheorem*{cor*}{Corollary}
\newtheorem*{dfn*}{Definition}

\newtheorem{cthm}{Theorem}

\newtheorem{cor}[equation]{Corollary}
\newtheorem{prop}[equation]{Proposition}

\theoremstyle{definition}

\newtheorem{rem}[equation]{Remark}

\newcommand{\inj}{\hookrightarrow}
\newcommand{\surj}{\twoheadrightarrow}

\newcommand{\opn}{\operatorname}
\newcommand{\cat}[1]{\operatorname{\mathsf{#1}}}

\newcommand{\mrm}[1]{\mathrm{#1}}

\newcommand{\projdim}{\operatorname{proj\,dim}}

\newcommand{\op}{\opn{op}}

\title{Gorenstein acyclic complexes and finitistic dimensions}
\author{Liran Shaul}


\begin{document}
\begin{abstract}
Given a two-sided noetherian ring $A$ with a dualizing complex,
we show that the big finitistic dimension of $A$ is finite if and only if every bounded below Gorenstein-projective-acyclic cochain complex of Gorenstein-projective $A$-modules is contractible.
If $A$ is further assumed to be an Artin algebra, we also prove a Gorenstein variant of a theorem of Rickard, showing its finitistic dimension is finite in case its Gorenstein-injective derived category is generated by the Gorenstein-injective modules.
\end{abstract}

\numberwithin{equation}{section}
\maketitle

\setcounter{section}{-1}
\section{Introduction}

Gorenstein homological algebra is a rapidly developing relative version of homological algebra in which the projective and injective modules are replaced by the more general Gorenstein-projective and Gorenstein-injective modules. 
The subject was initiated by Auslander, 
who proved a Gorenstein analogue of the fact that a commutative noetherian local ring is regular if and only if every finitely generated module has finite projective dimension. He showed that a noetherian local ring is Gorenstein if and only if every finitely generated module has a finite resolution made of Gorenstein-projective modules.

This paper is dedicated to studying the big finitistic dimension of a ring $A$,
which is defined to be
\[
\opn{FPD}(A) = \sup_{M \in \opn{Mod}(A), \projdim_A(M) <\infty} (\projdim_A(M)).
\]
A major open problem in homological algebra, known as the big finitistic dimension conjecture states that if $A$ is an Artin algebra then $\opn{FPD}(A) < \infty$.
In this paper we give Gorenstein homological versions of recent results from \cite{Rickard,ShFinDim,ShGor,Zhang} concerning the question when $\opn{FPD}(A) < \infty$.
It is often the case that results in ordinary homological algebra have analogues in Gorenstein homological algebra. Finding such analogues in the case of finitistic dimension is particularly attractive, because a result of Holm \cite[Theorem 2.28]{Holm} states that for any ring $A$, 
the Gorenstein version of the finitistic dimension coincides with the usual definition of $\opn{FPD}(A)$.

To state our results, recall that the notion of an acyclic complex has Gorenstein versions in Gorenstein homological algebra. A cochain complex $X$ is called a $\mathcal{GP}$-acyclic complex if for any Gorenstein-projective $A$-module $G$, it holds that $\opn{Hom}_A(G,X)$ is acyclic.
Every $\mathcal{GP}$-acyclic complex is acyclic, but the converse fails.
This notion is recalled and studied in more detail in \cref{sec:prem} below.

Our main result, repeated as \cref{thm:main} in the body of the paper states:
\begin{cthm}
Let $A$ be a ring,
and consider the following statements:
\begin{enumerate}
\item $\opn{FPD}(A) < \infty$.
\item Every bounded below $\mathcal{GP}$-acyclic complex of Gorenstein-projective $A$-modules is contractible.
\item Every bounded below acyclic complex of projective $A$-modules is contractible.
\end{enumerate}
Then it always holds that (1) $\implies$ (2) $\implies$ (3).
If $A$ is two-sided noetherian, and has a dualizing complex,
then (3) $\implies$ (1),
so that all three conditions are equivalent.
\end{cthm}

The fact that (1) $\implies$ (3), and the fact that (3) $\implies$ (1) when the ring is noetherian and admits a dualizing complexes are not new, and were proved in \cite{ShGor}.
The major new contribution of the result above is the introduction of condition (2),
and the fact that for two-sided noetherian rings with a dualizing complex (so in particular, for all Artin algebras), this condition is equivalent to $\opn{FPD}(A) < \infty$.

The condition (3) implies that the finitistic dimension conjecture is reflected in the category $\cat{K}_{\opn{ac}}(\opn{Proj}(A))$,
the homotopy category of acyclic complexes of projective modules,
which is a model for the singularity category of $A^{\op}$.
Condition (2) of the above result suggests that the data whether $\opn{FPD}(A) < \infty$ is also encoded in the category $\cat{K}_{\mrm{gpac}}(\mathcal{GP}(A))$,
the homotopy category of $\mathcal{GP}$-acyclic complexes of Gorenstein-projective $A$-modules,
which is a kind of a Gorenstein singularity category.

Our second main result is a Gorenstein version of a theorem of Rickard from \cite{Rickard}.
In that paper, answering a question of Keller, Rickard showed that for a finite dimensional algebra $A$ over a field, if the localizing subcategory generated by the injective $A$-modules is equal to $\cat{D}(A)$, then $\opn{FPD}(A) < \infty$. 
Our generation result concerns generation by the Gorenstein-injective modules,
but  we must replace the ordinary derived category by the Gorenstein-injective derived category $\cat{D}_{\mathcal{GI}}(A)$,
which was introduced in \cite{GaoZhang}, and in recalled in \cref{sec:prem} below.
Using this notion, in \cref{cor:gener} we show:

\begin{cthm}
Let $A$ be an Artin algebra,
and suppose that the localizing subcategory of the 
Gorenstein-injective derived category $\cat{D}_{\mathcal{GI}}(A)$
generated by the Gorenstein-injective $A$-modules is equal to $\cat{D}_{\mathcal{GI}}(A)$.
Then $\opn{FPD}(A) < \infty$.
\end{cthm}

One other result of interest is obtained as a byproduct of our methods.
To explain it, recall that an acyclic complex $P$ of projective $A$-modules is contractible if and only if all its cocycles $\ker(d^n_P)$ are projective $A$-modules.
In \cite[Theorem 8.6]{NeemanHo}, Neeman showed that this is the case if and only if all its cocycles $\ker(d^n_P)$ are flat $A$-modules.
In \cref{cor:cont-criteria} we obtain the following new contractiblity criteria for acyclic complexes of projectives, provided that they are bounded below:
\begin{cthm}
Let $A$ be a ring, 
and let $P$ be an acyclic bounded below complex of projective $A$-modules. Then $P$ is contractible if and only if all its cocycles $\ker(d^n_P)$ are Gorenstein-projective $A$-modules.
\end{cthm}

\section{Preliminaries}\label{sec:prem}

In this section we recall some basics about Gorenstein homological algebra and Gorenstein derived categories, and prove some new foundational results about the latter.
For more background about this subject the reader may consult \cite{Christensen,EnJen,GaoZhang,Holm2, Holm, Iacob} and their references.
Rings are assumed to be associative and unital,
modules are by default left modules,
and we grade complexes cohomologically.
A complex $M$ is called acyclic if it is exact,
that is, if $\mrm{H}^n(M) = 0$ for all $n \in \mathbb{Z}$.

\subsection{Gorenstein special modules}

Let $A$ be a ring.
Following \cite{EJDef},
an $A$-module $G$ is called a Gorenstein-projective $A$-module if there exist an acyclic complex of projective $A$-modules $P$,
such that $G \cong \ker(d^0_P)$,
and such that $\opn{Hom}_A(P,Q)$ is acyclic for any projective $A$-module $Q$.
An $A$-module $J$ is called a Gorenstein-injective $A$-module if there is an acyclic complex of injective $A$-modules $I$,
such that $J \cong \ker(d^0_I)$,
and such that $\opn{Hom}_A(E,I)$ is an acyclic complex for all injective $A$-modules $E$.
Finally, following \cite{EJFlat},
an $A$-module $H$ is called a Gorenstein-flat $A$-module if there is an acyclic complex of flat $A$-modules $F$ such that $H \cong \ker(d^0_F)$,
and such that $I\otimes_A F$ is an acyclic complex for all injective right $A$-modules $I$.

It is immediate from the definitions that any projective module is Gorenstein-projective, any injective module is Gorenstein-injective, and any flat module is Gorenstein-flat.
We will denote by $\mathcal{GP}(A)$ the category of Gorenstein-projective $A$-modules, and by $\mathcal{GI}(A)$ the category of Gorenstein-injective $A$-modules.

\subsection{Relative acyclic complexes}
Next, we recall some relative versions of acyclicity.
Let $\mathcal{P}$ be any non-empty class of $A$-modules which is closed under isomorphisms.
Recall that a complex $X$ of $A$-modules is called right $\mathcal{P}$-acyclic if for any $P \in \mathcal{P}$,
the complex $\opn{Hom}_A(P,X)$ is acyclic.
For example, if $\mathcal{P}$ is the class of all projective $A$-modules, then being right $\mathcal{P}$-acyclic is equivalent to being acyclic.
It follows that if $\mathcal{P}$ contains the class of all projective $A$-modules, 
then any right $\mathcal{P}$-acyclic complex is acyclic.
When $\mathcal{P} = \mathcal{GP}$,
the class of Gorenstein-projective modules,
such a complex is usually called a $\mathcal{GP}$-acyclic complex.
Given a ring $A$, 
the collection of all $\mathcal{GP}$-acyclic complexes over $A$ is a triangulated subcategory of the homotopy category $\cat{K}(A)$, 
which we will denote by $\cat{K}_{\mrm{gpac}}(A)$.
A map of complexes $X\to Y$ is called a $\mathcal{GP}$-quasi-isomorphism if its cone is $\mathcal{GP}$-acyclic.

Dually, for a non-empty class of objects $\mathcal{I}$ closed under isomorphisms, a complex $X$ will be called left $\mathcal{I}$-acyclic if for any $I \in \mathcal{I}$,
the complex $\opn{Hom}_A(X,I)$ is acyclic.
In the case where $\mathcal{I} = \mathcal{GI}$,
the class of Gorenstein-injective modules,
we will say that $X$ is a $\mathcal{GI}$-acyclic complex.
The collection of all $\mathcal{GI}$-acyclic complexes over a ring $A$ will be 
denoted by $\cat{K}_{\mrm{giac}}(A)$. 
Again, this is a triangulated subcategory.

\subsection{Relative derived categories}

Given a ring $A$, following \cite{GaoZhang},
we define the Gorenstein-projective derived category of $A$,
denoted by $\cat{D}_{\mathcal{GP}}(A)$,
to be the Verdier quotient 
\[
\cat{D}_{\mathcal{GP}}(A) := \cat{K}(A) / \cat{K}_{\mrm{gpac}}(A).
\]
This is a triangulated category, studied in detail in \cite{GaoZhang}.
Dually, the Gorenstein-injective derived category of $A$ will be denoted by  $\cat{D}_{\mathcal{GI}}(A)$.
It is given as the Verdier quotient
\[
\cat{D}_{\mathcal{GI}}(A) := \cat{K}(A) / \cat{K}_{\mrm{giac}}(A).
\]

Of special importance in this paper is the problem whether these categories admit products and coproducts. We shall now discuss this problem. 
We note that this is far from trivial, see \cite[Example 4.1]{NeeExa} for a similar situation in which products fail to exist.

\begin{prop}\label{prop:hasprod}
Let $A$ be a ring. Then the following hold:
\begin{enumerate}
\item A product of any collection of $\mathcal{GP}$-acyclic complexes over $A$ is $\mathcal{GP}$-acyclic.
\item A coproduct of any collection of $\mathcal{GI}$-acyclic complexes over $A$ is $\mathcal{GI}$-acyclic.
\end{enumerate}
\end{prop}
\begin{proof}
Let $\{X_i\}_{i \in I}$ be a collection of $\mathcal{GP}$-acyclic complexes over $A$,
and let $P$ be a Gorenstein-projective $A$-module.
We have that
\[
\opn{Hom}_A(P,\prod_{i \in I} X_i) \cong \prod_{i \in I} \opn{Hom}_A(P,X_i).
\]
For each $i \in I$,
since $X_i$ is an $\mathcal{GP}$-acyclic complex,
by definition $\opn{Hom}_A(P,X_i)$ is an acyclic complex.
Since the product of acyclic complexes is acyclic,
the above isomorphism shows that $\opn{Hom}_A(P,\prod_{i \in I} X_i)$ is acyclic,
which implies that $\prod_{i \in I} X_i$ is $\mathcal{GP}$-acyclic.
The second statement is proved similarly.
\end{proof}

\begin{cor}\label{cor:has-prod}
Let $A$ be a ring. Then the following hold:
\begin{enumerate}
\item The Gorenstein-projective derived category $\cat{D}_{\mathcal{GP}}(A)$ has arbitrary products.
\item The Gorenstein-injective derived category $\cat{D}_{\mathcal{GI}}(A)$ has arbitrary coproducts.
\end{enumerate}
\end{cor}
\begin{proof}
The first statement follows from \cref{prop:hasprod}(1) and the dual of \cite[Lemma 1.5]{BokNee}.
The second statement follows from \cref{prop:hasprod}(2) and \cite[Lemma 1.5]{BokNee}.
\end{proof}

We do not know if the dual statements hold. 
The next result given a necessary and sufficient condition for them to hold:

\begin{prop}
Let $A$ be a ring.
\begin{enumerate}
\item The Gorenstein-projective derived category $\cat{D}_{\mathcal{GP}}(A)$ has arbitrary coproducts which coincide with the coproducts in $\cat{K}(A)$ if and only if any coproduct of a collection of $\mathcal{GP}$-acyclic complexes over $A$ is $\mathcal{GP}$-acyclic.
\item The Gorenstein-injective derived category $\cat{D}_{\mathcal{GI}}(A)$ has arbitrary products which coincide with the products in $\cat{K}(A)$ if and only if any product of a collection of $\mathcal{GI}$-acyclic complexes over $A$ is $\mathcal{GI}$-acyclic.
\end{enumerate}
\end{prop}
\begin{proof}
If a coproduct of $\mathcal{GP}$-acyclic complexes is $\mathcal{GP}$-acyclic,
then it follows from \cite[Lemma 1.5]{BokNee} that $\cat{D}_{\mathcal{GP}}(A)$ has arbitrary coproducts.
Conversely,
if $\cat{D}_{\mathcal{GP}}(A)$ has arbitrary coproducts,
and $\{X_i\}_{i \in I}$ is a collection of $\mathcal{GP}$-acyclic complexes over $A$, 
then in $\cat{D}_{\mathcal{GP}}(A)$ it holds that each $X_i \cong 0$.
Set $T:= \oplus_{i \in I} X_i$.
Then $T \cong \oplus_{i \in I} 0 \cong 0$ in $\cat{D}_{\mathcal{GP}}(A)$.
Let $P$ be a Gorenstein-projective $A$-module.
Then for all $n$ we have that:
\[
0 = \opn{Hom}_{\cat{D}_{\mathcal{GP}}(A)}(P,0[n]) = \opn{Hom}_{\cat{D}_{\mathcal{GP}}(A)}(P,T[n]).
\]
According to \cite[Proposition 2.8]{GaoZhang},
we have that
\[
0 = \opn{Hom}_{\cat{D}_{\mathcal{GP}}(A)}(P,T[n]) \cong \opn{Hom}_{\cat{K}(A)}(P,T[n]),
\]
and as the latter by definition is equal to $\mrm{H}^n(\opn{Hom}_A(P,T))$,
we deduce that $\opn{Hom}_A(P,T)$ is acyclic,
so that $T$ is $\mathcal{GP}$-acyclic as claimed.
The second statement is proven similarly.
\end{proof}

\subsection{Dualizing complexes}

The class of rings for which the main results of this paper hold are the two-sided noetherian rings which admit a dualizing complex.
These were first introduced in noncommutative algebra by Yekutieli in \cite{Yek}. In this paper, following \cite{InKr},
we need a slight variant of Yekutieli's definition, and define a dualizing complex $R$ over a ring $A$ to be a bounded complex of $A$-$A$-bimodules,
which are injective both as $A$-modules and as $A^{\op}$-modules,
have finitely generated cohomology both on the left and on the right, and such that the natural maps $A\to \opn{Hom}_A(R,R)$ and $A\to \opn{Hom}_{A^{\op}}(R,R)$ are quasi-isomorphisms.
In particular, any Artin algebra $A$ admits a dualizing complex.
In fact, in that case it is concentrated in a single degree,
so we refer to it as the dualizing module of $A$.
A foundational result in Gorenstein homological algebra \cite{JorRes}, due to J{\o}rgensen, states that over a two-sided noetherian ring $A$ with a dualizing complex, for any $A$-module $M$,
there exist a $\mathcal{GP}$-quasi-isomorphism $P \to M$,
where $P$ is a bounded above complex of Gorenstein-projective $A$-modules. Such a $P$ is called a proper Gorenstein-projective resolution of $M$.

\section{Contractible complexes of Gorenstein-projectives}

In this section we show that certain relative acyclic complexes must be contractible over rings of finite finitistic dimension,
establishing a relative version of the results of \cite[Appendix A]{ShGor}.

It is well known that an acyclic complex of projective modules is contractible if and only if all its cocycles are projective.
The next result is a relative version of this fact.

\begin{prop}\label{prop:contr}
Let $A$ be a ring, and let $\mathcal{P}$ be any non-empty class of $A$-modules which is closed under isomorphisms and direct summands,
and contains all projective $A$-modules.
Let $X$ be a right $\mathcal{P}$-acyclic complex,
such that $X^n \in \mathcal{P}$ for all $n \in \mathbb{Z}$.
Then $X$ is contractible if and only if all its cocycles $\ker(d^n_X)$ belong to $\mathcal{P}$.
\end{prop}
\begin{proof}
Suppose first that $X$ is contractible.
This implies that for any $n$,
the short exact sequence
\[
0 \to \ker(d^n_X) \to X^n \to \ker(d^{n+1}_X) \to 0
\]
is split, so in particular $\ker(d^n_X)$ is a direct summand of $X^n$, which implies that $\ker(d^n_X) \in \mathcal{P}$.
Conversely, suppose that for all $n$, 
we have that $\ker(d^n_X) \in \mathcal{P}$.
Since by assumption $X$ is acyclic,
to show it is contractible,
it is enough to show that the sequence
the short exact sequence
\begin{equation}\label{eqn:to-split}
0 \to \ker(d^n_X) \to X^n \to \ker(d^{n+1}_X) \to 0
\end{equation}
splits.
To see this, we apply the functor $\opn{Hom}_A(\ker(d^{n+1}_X),-)$ to this sequence,
and we claim that the resulting sequence
\begin{equation}\label{eqn:Hom-exact}
0 \to \opn{Hom}_A(\ker(d^{n+1}_X),\ker(d^n_X)) \to \opn{Hom}_A(\ker(d^{n+1}_X),X^n) \to \opn{Hom}_A(\ker(d^{n+1}_X),\ker(d^{n+1}_X)) \to 0
\end{equation}
is still exact.
To see this, by left exactness of the covariant hom functor,
it is enough for us to show that the map
\[
\opn{Hom}_A(\ker(d^{n+1}_X),X^n) \to \opn{Hom}_A(\ker(d^{n+1}_X),\ker(d^{n+1}_X))
\]
is surjective.
Since by assumption $X$ is right $\mathcal{P}$-acyclic,
and $\ker(d^{n+1}_X) \in \mathcal{P}$,
it follows that the cochain complex $\opn{Hom}_A(\ker(d^{n+1}_X),X)$ is acyclic.
In particular, 
the sequence
\begin{equation}\label{eqn:exact-X}
\opn{Hom}_A(\ker(d^{n+1}_X),X^n) \to \opn{Hom}_A(\ker(d^{n+1}_X),X^{n+1}) \to 
\opn{Hom}_A(\ker(d^{n+1}_X),X^{n+2})
\end{equation}
is exact.
Given a map $\alpha \in \opn{Hom}_A(\ker(d^{n+1}_X),\ker(d^{n+1}_X))$,
since $\ker(d^{n+1}_X) = \opn{Im}(d^n_X) \subseteq X^{n+1}$,
we may consider $\alpha$ as an element of $\opn{Hom}_A(\ker(d^{n+1}_X),X^{n+1})$.
The map 
\[
\opn{Hom}_A(\ker(d^{n+1}_X),X^{n+1}) \xrightarrow{(d^{n+1}_X)_*}
\opn{Hom}_A(\ker(d^{n+1}_X),X^{n+2})
\]
is given by composition with $d^{n+1}_X$,
which shows that $(d^{n+1}_X)_*(\alpha) = d^{n+1}_X\circ \alpha = 0$.
Thus, exactness of \cref{eqn:exact-X}
shows that $\alpha$ is in the image of $(d^n_X)_*$,
and hence \cref{eqn:Hom-exact} is exact.
In particular, there exist $\beta:\ker(d^{n+1}_X) \to X^n$,
such that $d^n_X \circ \beta = 1_{\ker(d^{n+1}_X)}$,
which shows that \cref{eqn:to-split} splits.
\end{proof}

Applying \cref{prop:contr} to the class of Gorenstein-projective modules, we obtain:

\begin{cor}\label{cor:GP-cont}
Let $A$ be a ring, and let $X$ be a $\mathcal{GP}$-acyclic complex of Gorenstein-projective modules.
Then $X$ is contractible if and only if all its cocycles $\ker(d^n_X)$ are Gorenstein-projective. 
\end{cor}

Here is the main result of this section.

\begin{thm}\label{thm:FPDimplyCont}
Let $A$ be a ring such that $\opn{FPD}(A) < \infty$.
Then any bounded below $\mathcal{GP}$-acyclic complex of Gorenstein-projective $A$-modules is contractible.
\end{thm}
\begin{proof}
Suppose $\opn{FPD}(A) = s < \infty$.
According to \cite[Theorem 2.28]{Holm},
there is an equality
\[
\opn{FPD}(A) = \opn{FGPD}(A),
\]
where $\opn{FGPD}(A)$ is the finitistic Gorenstein projective dimension of $A$.
Let $X$ be a bounded below $\mathcal{GP}$-acyclic complex of Gorenstein-projective modules.
Since any projective module is Gorenstein-projective,
in particular $X$ is acyclic.
The fact that $X$ is bounded below implies that
\[
0 \to \dots \to X^{n+s-1} \to X^{n+s} \xrightarrow{d^{n+s}_X} \opn{Im}(d^{n+s}_X)
\]
is a Gorenstein-projective resolution of the $A$-module $\opn{Im}(d^{n+s}_X)$, 
which shows that $\opn{Im}(d^{n+s}_X)$ has finite Gorenstein-projective dimension. 
Hence, it has Gorenstein-projective dimension at most $s$,
so by \cite[Theorem 2.20]{Holm},
its $s$-syzygy, which is simply $\ker(d^n_X)$ must be Gorenstein-projective.
\cref{cor:GP-cont} then implies that $X$ is contractible.
\end{proof}

\begin{rem}
One may easily obtain results dual to the results of this section,
and show,
for instance,
that if a ring has finite finitistic injective dimension,
then any bounded above $\mathcal{GI}$-acyclic complex of Gorenstein-injective $A$-modules is contractible.
\end{rem}

\section{Bounded below acyclic complexes of Gorenstein-projectives}

The aim of this section is to obtain a converse of \cref{thm:FPDimplyCont},
showing that contractibility of certain complexes of Gorenstein-projectives implies finite finitistic dimension.

The next result is similar to \cite[Remark 3.3]{Holm2}.
Because of its centrality to the argument below, 
we include full details.

\begin{prop}\label{prop:flat-is-gp}
Let $A$ be a ring,
and let $X$ be a bounded acyclic complex of $A$-modules.
Suppose that for all $n \in \mathbb{Z}$,
it holds that the $A$-module $X^n$ has finite projective dimension over $A$. Then $X$ is $\mathcal{GP}$-acyclic.
\end{prop}
\begin{proof}
Let $G$ be some Gorenstein-projective $A$-module,
and assume $X$ is of the form
\begin{equation}\label{eqn:proj-res}
0 \to P^{-n} \to \dots P^{-1} \to P^0 \to 0.
\end{equation}
Applying the functor $\opn{Hom}_A(G,-)$ to the short exact sequence
\[
0 \to \ker(d^{-1}) \to P^{-1} \xrightarrow{d^{-1}} P^0 \to 0
\]
yields
\[
0 \to \opn{Hom}_A(G,\ker(d^{-1})) \to \opn{Hom}_A(G,P^{-1}) \to \opn{Hom}_A(G,P^0) \to \opn{Ext}^1_A(G,\ker(d^{-1})).
\]
The resolution (\ref{eqn:proj-res}) shows that $\ker(d^{-1})$ is isomorphic to a bounded complex of modules of finite projective dimension, and hence that $\projdim_A(\ker(d^{-1})) < \infty$,
so by \cite[Proposition 2.3]{Holm} we have that $\opn{Ext}^1_A(G,\ker(d^{-1})) = 0$,
showing that the map $ \opn{Hom}_A(G,P^{-1}) \to \opn{Hom}_A(G,P^0)$ is surjective.
Suppose we know that the sequence
\[
\opn{Hom}_A(G,P^{-i}) \to \opn{Hom}_A(G,P^{-(i-1)}) \to \dots \to \opn{Hom}_A(G,P^0) \to 0
\]
is exact,
and let $K = \opn{Im}(d^{-(i+1)}:P^{-(i+1)} \to P^{-i})$.
There is a short exact sequence
\[
0 \to N \to P^{-(i+1)} \to K \to 0, 
\]
and as $N$ is a syzygy of (\ref{eqn:proj-res}),
by the same argument as above, 
it holds that $\projdim_A(N) < \infty$.
Thus, using \cite[Proposition 2.3]{Holm} again,
the sequence
\[
0 \to \opn{Hom}_A(G,N) \to \opn{Hom}_A(G,P^{-(i+1)}) \to \opn{Hom}_A(G,K) \to 0 
\]
is exact, and as the map $d^{-(i+1)}:P^{-(i+1)} \to P^{-i}$ factors as $P^{-(i+1)} \surj K \inj P^{-i}$, it follows that $\opn{Hom}_A(G,P^{-(i+1)}) \to \opn{Hom}_A(G,P^{-i}) \to \opn{Hom}_A(G,P^{-(i-1)})$ is exact.
Induction on the amplitude of $X$ now shows that $X$ is $\mathcal{GP}$-acyclic.
\end{proof}

Next, we extend the above proposition to bounded below complexes.

\begin{thm}
Let $A$ be a ring,
and let $X$ be a bounded below acyclic complex,
with the property that $\projdim_A(X^n) < \infty$ for all $n \in \mathbb{Z}$.
Then $X$ is $\mathcal{GP}$-acyclic.    
\end{thm}
\begin{proof}
As $X$ is bounded below, there exist some $r \in \mathbb{Z}$ such that $X^n = 0$ for all $n<r$.
Let $G$ be a Gorenstein-projective $A$-module.
We must show that the complex $\opn{Hom}_A(G,X)$ is acyclic.
For any $n \in \mathbb{Z}$,
the fact that $X$ is bounded below implies the existence of an exact sequence of $A$-modules
\[
0 \to X^r \to X^{r+1} \to \dots \to X^{n-1} \to X^n \to X^{n+1} \xrightarrow{d^{n+1}_X} \opn{Im}(d^{n+1}_X) \to 0
\]
As by assumption each $X^i$ has finite projective dimension,
and as this sequence is exact, 
it follows that $\projdim_A(\opn{Im}(d^{n+1}_X) < \infty$.
Thus, we are in the situation of \cref{prop:flat-is-gp},
and it follows from it that the sequence
\begin{gather*}
0 \to \opn{Hom}_A(G,X^r) \to \opn{Hom}_A(G,X^{r+1}) \to \dots \to \opn{Hom}_A(G,X^{n-1}) \to\\ \opn{Hom}_A(G,X^n) \to \opn{Hom}_A(G,X^{n+1}) \to \opn{Hom}_A(G,\opn{Im}(d^{n+1}_X)) \to 0
\end{gather*}
is exact.
In particular, for any $n \in \mathbb{Z}$,
we see that the sequence
\[
\opn{Hom}_A(G,X^{n-1}) \to \opn{Hom}_A(G,X^n) \to \opn{Hom}_A(G,X^{n+1})
\]
is exact,
which shows that $\opn{Hom}_A(G,X)$ is acyclic,
and hence that $X$ is $\mathcal{GP}$-acyclic.
\end{proof}

\begin{cor}\label{cor:proj-is-gp-acyclic}
Over any ring, any acyclic bounded below complex of projective $A$-modules is $\mathcal{GP}$-acyclic.
\end{cor}

Since projective modules are Gorenstein-projectives,
combining this with \cref{cor:GP-cont} we obtain:
\begin{cor}\label{cor:cont-criteria}
Let $A$ be a ring, 
and let $P$ be an acyclic bounded below complex of projective $A$-modules. Then $P$ is contractible if and only if all its cocycles $\ker(d^n_P)$ are Gorenstein-projective $A$-modules.
\end{cor}

Here is the main result of this section and of this paper:

\begin{thm}\label{thm:main}
Let $A$ be a ring,
and consider the following statements:
\begin{enumerate}
\item $\opn{FPD}(A) < \infty$.
\item Every bounded below $\mathcal{GP}$-acyclic complex of Gorenstein-projective $A$-modules is contractible.
\item Every bounded below acyclic complex of projective $A$-modules is contractible.
\end{enumerate}
Then it always holds that (1) $\implies$ (2) $\implies$ (3).
If $A$ is two-sided noetherian, and has a dualizing complex,
then (3) $\implies$ (1),
so that all three conditions are equivalent.
\end{thm}
\begin{proof}
The fact that (1) $\implies$ (2) is \cref{thm:FPDimplyCont},
and that (2) $\implies$ (3) follows from \cref{cor:proj-is-gp-acyclic}.
Finally, if $A$ is two-sided noetherian and has a dualizing complex,
then by \cite[Theorem 5.1]{ShGor}, 
we have that (3) $\implies$ (1).
\end{proof}

\begin{rem}
In general, it could happen that (3) holds but (1) fails.
See \cite[Theorem 5.1(d)]{ShGor} for details.
We do not know if for an arbitrary ring (3) $\implies$ (2).
See \cite{Posi} for some examples of (noncommutative non-noetherian) rings over which (3) fails.
\end{rem}

\section{Gorenstein-injective generation and acyclic complexes}

In this section we give a sufficient condition for the finitistic dimension conjecture to hold over Artin algebras.
Over an Artin algebra, the flat and projective modules coincide,
so it might be interesting to note (although this result is not used in the sequel) that the same is true for Gorenstein-projectives and Gorenstein flat modules:
\begin{prop}
Over an Artin algebra $A$,
the classes of Gorenstein-projective and Gorenstein-flat modules coincide.
\end{prop}
\begin{proof}
Since over an Artin algebra $A$, flat modules are projective,
given an acyclic complex $P$ of projective $A$-modules,
it is thus enough to show that $\opn{Hom}_A(P,Q)$ is exact for all left projective $A$-modules $Q$ if and only if $I\otimes_A P$ is exact for all right injective $A$-modules $I$.
To see this, note that if $R$ is the dualizing module of $A$,
then any such $I$ is a direct summand of direct sums of copies of $R$,
so it is enough to show that $\opn{Hom}_A(P,Q)$ is exact for all left projective $A$-modules $Q$ if and only if $R\otimes_A P$ is exact,
and this is exactly \cite[Lemma 1.7]{JorRes}.
\end{proof}

In \cite[Theorem 4.8]{InKr},
the authors constructed the covariant Grothendieck duality,
which is an equivalence $\cat{K}(\opn{Proj}(A)) \to \cat{K}(\opn{Inj}(A))$
over any two-sided noetherian ring with a dualizing complex.
The next result is a Gorenstein version of this result for Artin algebras.

\begin{prop}\label{prop:GrothEquiv}
Let $A$ be an Artin algebra with a dualizing module $R$.
\begin{enumerate}
\item The functor 
\[
R\otimes_A -: \mathcal{GP}(A) \to \mathcal{GI}(A)
\]
is an equivalence of categories with a quasi-inverse 
\[
\opn{Hom}_A(R,-):\mathcal{GI}(A) \to \mathcal{GP}(A).
\]`
\item The functor 
\[
R\otimes_A -: \cat{K}(\mathcal{GP}(A)) \to \cat{K}(\mathcal{GI}(A))
\]
is an equivalence of categories with a quasi-inverse 
\[
\opn{Hom}_A(R,-):\cat{K}(\mathcal{GI}(A)) \to \cat{K}(\mathcal{GP}(A)).
\]
\end{enumerate}
\end{prop}
\begin{proof}
Since $R\otimes_A -$ and $\opn{Hom}_A(R,-)$ are additive functors,
clearly the second statement follows from the first one.
The first statement is contained in \cite[Proposition X.1.4]{BelReit}.
\end{proof}

Given a triangulated category $\mathcal{T}$ which has infinite coproducts, and given a non-empty set of objects $\mathbf{S}$ in $\mathcal{T}$, 
recall that the localizing subcategory generated by $\mathbf{S}$,
denoted by $\opn{Loc}_{\mathcal{T}}(\mathbf{S})$ is the smallest triangulated subcategory of $\mathcal{T}$ which contains $\mathbf{S}$ and is closed under infinite coproducts.
In case $\opn{Loc}_{\mathcal{T}}(\mathbf{S}) = \mathcal{T}$,
we say that $\mathbf{S}$ generates $\mathcal{T}$.

We have seen in \cref{cor:has-prod} that over any ring $A$,
the Gorenstein-injective derived category 
$\cat{D}_{\mathcal{GI}}(A)$ has arbitrary coproducts,
so it makes sense to talk about localizing subcategories of it.
In \cite[Theorem 4.3]{Rickard},
it was shown that for a finite dimensional algebra $A$ over a field,
if $\opn{Loc}_{\cat{D}(A)}(\opn{Inj}(A)) = \cat{D}(A)$,
then $\opn{FPD}(A)<\infty$.
This result was further generalized to larger classes of rings in \cite[Theorem 1.1]{Zhang} and \cite[Corollary 5.3]{ShFinDim}. 
Our next result (or rather a corollary of it) is a Gorenstein version of this fact.

\begin{thm}\label{thm:generation}
Let $A$ be an Artin algebra,
and suppose that 
\[
\opn{Loc}_{\cat{D}_{\mathcal{GI}}(A)}(\mathcal{GI}(A)) = \cat{D}_{\mathcal{GI}}(A).
\]
Then any bounded below $\mathcal{GP}$-acyclic complex of Gorenstein-projective modules is contractible.
\end{thm}
\begin{proof}
Let $P$ be a bounded below $\mathcal{GP}$-acyclic complex of Gorenstein-projective modules.
Let $n \in \mathbb{Z}$, 
and let $I$ be some Gorenstein-injective $A$-module.
By \cref{prop:GrothEquiv}, $F = \opn{Hom}_A(R,I)$ is a Gorenstein-Projective $A$-module.
The fact that $P$ is $\mathcal{GP}$-acyclic implies that
\[
\opn{Hom}_{\cat{D}_{\mathcal{GP}}(A)}(F[n],P) = 0
\]
By \cite[Proposition 2.8]{GaoZhang},
there is an equality
\[
\opn{Hom}_{\cat{D}_{\mathcal{GP}}(A)}(F[n],P) = \opn{Hom}_{\cat{K}(A)}(F[n],P), 
\]
and by \cref{prop:GrothEquiv},
we have that
\[
\opn{Hom}_{\cat{K}(A)}(F[n],P) = \opn{Hom}_{\cat{K}(A)}(R\otimes_A F[n],R\otimes_A P) = \opn{Hom}_{\cat{K}(A)}(I[n],R\otimes_A P).
\]
Since $P$ is a bounded-below complex of Gorenstein-projectives,
it follows from \cref{prop:GrothEquiv} that $R\otimes_A P$ is a bounded-below complex of Gorenstein-injectives.
By the dual result of \cite[Proposition 2.8]{GaoZhang},
we obtain that
\[
0 =  \opn{Hom}_{\cat{K}(A)}(I[n],R\otimes_A P) = \opn{Hom}_{\cat{D}_{\mathcal{GI}}(A)}(I[n],R\otimes_A P).
\]
Consider the following subcategory of $\cat{D}_{\mathcal{GI}}(A)$:
\[
\mathcal{T} = \{X \in \cat{D}_{\mathcal{GI}}(A) \mid \forall n \in \mathbb{Z}, \opn{Hom}_{\cat{D}_{\mathcal{GI}}(A)}(X[n],R\otimes_A P) = 0 \}.
\]
It is clear that $\mathcal{T}$ is a localizing subcategory of $\cat{D}_{\mathcal{GI}}(A)$,
and we have seen that it contains all Gorenstein-injective $A$-modules.
Hence, by the generation assumption,
it follows that $\mathcal{T} = \cat{D}_{\mathcal{GI}}(A)$,
so in particular $R\otimes_A P \in \mathcal{T}$,
which shows that $R\otimes_A P \cong 0$ in $\cat{D}_{\mathcal{GI}}(A)$.
Since it is a bounded-below complex of Gorenstein-injectives,
again, the dual of \cite[Proposition 2.8]{GaoZhang} implies that
$R\otimes_A P \cong 0$ in $\cat{K}(A)$,
and hence, \cref{prop:GrothEquiv} shows that
\[
P \cong \opn{Hom}_A(R,R\otimes_A P) \cong 0
\]
in $\cat{K}(A)$, as claimed.
\end{proof}

\begin{cor}\label{cor:gener}
Let $A$ be an Artin algebra,
and suppose that 
\[
\opn{Loc}_{\cat{D}_{\mathcal{GI}}(A)}(\mathcal{GI}(A)) = \cat{D}_{\mathcal{GI}}(A).
\]
Then $\opn{FPD}(A) < \infty$.    
\end{cor}
\begin{proof}
This follows from \cref{thm:generation} and \cref{thm:main}.
\end{proof}

\bibliographystyle{abbrv}
\bibliography{main}

\end{document}